\documentclass[12pt]{amsart}
\pdfoutput=1
\usepackage{heuristica}
\usepackage[heuristica,vvarbb,bigdelims]{newtxmath}

\usepackage{fullpage,url,mathrsfs,amssymb,color,multicol,enumerate,amscd}
\usepackage{multirow}
\usepackage{mathtools}
\usepackage[all]{xy} 
\usepackage{verbatim,moreverb}

\usepackage[OT2,T1]{fontenc}

\usepackage{tikz-cd}
\usepackage{xfrac}
\usepackage{float}

\definecolor{backgroundcolor}{rgb}{0.9,0.8,0.8}

\numberwithin{equation}{section}
\usepackage{setspace}
\usepackage{color}
\newcommand{\vanish}[1]{}

\def\bbar#1{\setbox0=\hbox{$#1$}\dimen0=.2\ht0 \kern\dimen0 }
\usepackage{enumitem}

\newenvironment{romanenum}{\hfill \begin{enumerate}[label=({\roman*})]
}
{\end{enumerate}}
\DeclareSymbolFont{cyrletters}{OT2}{wncyr}{m}{n}
\DeclareMathSymbol{\Sha}{\mathalpha}{cyrletters}{"58}

\newcommand{\G}{{\mathbb G}}

\newcommand{\PP}{{\mathbb P}}
\newcommand{\Q}{{\mathbb Q}}

\newcommand{\ZZ}{{\mathbb Z}}
\newcommand{\Zhat}{{\hat{\ZZ}}}
\newcommand{\kbar}{{\bbar{k}}}

\def\bbar#1{\setbox0=\hbox{$#1$}\dimen0=.2\ht0 \kern\dimen0 \overline{\kern-\dimen0 #1}}
\newcommand{\Qbar}{{\overline{\mathbb Q}}} 
\newcommand{\Qab}{{\mathbb Q}^{\operatorname{ab}}}

\newcommand{\calA}{{\mathcal A}}

\newcommand{\calI}{{\mathcal I}}

\DeclareMathOperator{\Aut}{Aut}
\DeclareMathOperator{\Gal}{Gal}

\newcommand{\GalQab}{{\Gal}(\Qbar/\Qab)}
\newcommand{\GalQ}{{\Gal}(\Qbar/\Q)}
\newcommand{\Galk}{{\Gal}(\kbar/k)}

\newcommand{\GL}{\operatorname{GL}}
\newcommand{\SL}{\operatorname{SL}}

\newcommand{\PSL}{\operatorname{PSL}}

\newcommand{\cyc}{\operatorname{cyc}}
\newcommand{\Qcyc}{{\Q^{\cyc}}} 

\newcommand{\GalQcyc}{{\Gal}(\Qbar/\Qcyc)}
\newcommand{\normsbgp}{\vartriangleleft}

\newcommand{\isom}{\simeq}

\newcommand{\intersect}{\cap} 
\newcommand{\Union}{\bigcup} 

\newcommand{\Quo}{\operatorname{Quo}}

\usepackage{amsthm}

\definecolor{webcolor}{rgb}{0,0,1}
\definecolor{webbrown}{rgb}{.6,0,0}
\usepackage[
        colorlinks,
        linkcolor=webbrown,  filecolor=webbrown,  citecolor=webbrown, 
        backref,
        pdfauthor={Rakvi}, 
        pdftitle={},
]{hyperref}
\usepackage{cleveref}
\showboxdepth=\maxdimen
\showboxbreadth=\maxdimen
\theoremstyle{definition}
\newtheorem{theorem}{Theorem}[section]
\newtheorem{lemma}[theorem]{Lemma}
\newtheorem{corollary}[theorem]{Corollary}
\newtheorem{proposition}[theorem]{Proposition}
\newtheorem{definition}[theorem]{Definition}

\newtheorem*{Mazur's Program B}{Mazur's Program B}

\theoremstyle{remark}
\newtheorem*{remark}{Remark}

\crefname{section}{§}{§§}
\crefname{lemma}{Lemma}{Lemmas}
\crefname{equation}{equation}{equations}
\crefname{theorem}{Theorem}{Theorems}
\crefname{proposition}{Proposition}{Propositions}

\usepackage[alphabetic,backrefs,lite]{amsrefs} 
\usepackage{datetime}
\date{\today}
\usepackage{amsfonts}
\usepackage{hyperref}
\usepackage{microtype}
\begin{document}

\title[]{On Possible Genus $0$ adelic Galois Images of non CM Elliptic Curves over $\Q$}
\author{Rakvi}

\begin{abstract}

Let $E$ be an elliptic curve defined over $\Q.$ Associated to $E$, there is an adelic Galois representation $\rho_E \colon \GalQ \to \GL_2(\Zhat).$ In this article, we give possibilities for groups generated by $\rho_E(\GalQ)$ and $-I$ as $E$ varies over non CM elliptic curves defined over $\Q$ such that the group $\pm \rho_E(\GalQ)$ has genus 0. However, our list is not minimal.
 
\end{abstract}
\maketitle
\setcounter{tocdepth}{1}

\section{Introduction}
Let $E$ be a non-CM elliptic curve defined over $\Q$. Fix an algebraic closure $\overline{\Q}$ of $\Q$. For a positive integer $N$, let $E[N]$ denote the $N$-torsion subgroup of $E(\overline{\Q}).$ As a $\ZZ/N\ZZ$-module $E[N]$ is free of rank 2. Choosing a basis of $E[N]$, the natural action of $\G_{\Q} \coloneqq \GalQ$ on $E[N]$ gives a representation,
\[ \rho_{E,N}\colon \G_{\Q} \to \Aut(E[N]) \simeq \GL_{2}(\ZZ/N \ZZ). \] By choosing compatible bases for $E[N]$ with $N \ge 1$, we get a representation
\[ \rho_{E}\colon \G_{\Q} \to \GL_{2}(\hat{\ZZ}) \] such that the composition of $\rho_E$ with the reduction modulo $N$ map $\GL_{2}(\hat{\ZZ}) \rightarrow \GL_{2}(\ZZ/N \ZZ)$ is $\rho_{E,N}.$
The aim of this article is to give a list of groups generated by $\rho_E(\GalQ)$ and $-I$ as $E$ varies over non CM elliptic curves $E$ defined over $\Q$ such that the modular curve associated to the group $\pm \rho_E(\GalQ)$ has genus $0.$ By $G^t$ we mean the group generated by transpose of members of $G.$

\subsection{Constraints on the Galois image} Serre's open image theorem \cite{MR0387283} tells us that the group $\rho_{E}(\G_{\Q})$ is open and hence of finite index in $\GL_{2}(\hat{\ZZ})$. Let $\chi_{\cyc} \colon \G_{\Q} \to \hat{\ZZ}^{\times}$ be the cyclotomic character that satisfies $\sigma(\zeta_n)=\zeta_n^{\chi_{\cyc}(\sigma)\mod n}$ for any integer $n \ge 1$, $n$-th root of unity $\zeta_n \in \Qbar$ and $\sigma \in \G_{\Q}$ and $\chi_{\cyc}=\det \circ \rho_{E}.$ Hence, \begin{equation}\label{eq:detsurjective}
    \det(\rho_{E}(\G_{\Q}))=\Zhat^{\times}.
\end{equation} By the Kronecker-Weber theorem, we also know that \begin{equation}\label{eq:commthick}
    [\rho_E(\G_{\Q}),\rho_E(\G_{\Q})]=\rho_E(\G_{\Q}) \intersect \SL_2(\Zhat)
\end{equation} because \[[\rho_E(\G_{\Q}),\rho_E(\G_{\Q})]=\rho_E(\GalQab)\] and \[\rho_E(\G_{\Q}) \intersect \SL_2(\Zhat)=\rho_E(\GalQcyc).\]

Let $G$ be a subgroup of $\GL_{2}(\hat{\ZZ})$ such that 
\begin{itemize}
     \item It is an open subgroup of $\GL_2(\Zhat).$
     \item It has full determinant, i.e., $\det(G)=\Zhat^{\times}.$
     \item The associated modular curve $(X_G,\pi_G)$ has genus $0$ and has a rational point, i.e., $X_G(\Q) \neq \emptyset.$
 \end{itemize}\label{conditions}

We represent our groups via modular curves associated to them, i.e., we describe morphisms $\pi_G$ explicitly and given $\pi_G$ a set of generators for group $G$ can be computed.

Define the following rational functions in $\Q(t).$ Here $v$ is a parameter that takes values in $\Q.$ 

\[\pi_1(t)= \scriptstyle {t^2}\]\[ \pi_{1,v}(t)=\scriptstyle vt^2\]

\[\pi_2(t)=\tfrac{t^2+\alpha}{t}, \pi_{2,v}(t)=\tfrac{vt^{2}-4\alpha t+\alpha t}{-t^{2}+vt-\alpha}\]

\[\pi_3(t)=\tfrac{t^3-3t+1}{t^2-t}\]
\[\pi_{3,v}(t)=\tfrac{(-v + 3)t^{3} + (-3v^{2} + 9v - 9)t^{2} + (-3v^{3} + 9v^{2} - 15v)t + (-v^{4} + 3v^{3} - 6v^{2} - v + 3)}{t^{3} + 2vt^{2} + (v^{2} + v - 3)t + (v^{2} - 3v + 1)}\]

\[\pi_4(t)=\tfrac{t^{4}-6t^{2}+1}{t^3-t}\]
\[\pi_{4,v}(t)=\tfrac{-vt^{4} + (8v + 16)t^{3} + (-18v - 96)t^{2} + (8v + 176)t + (7v - 96)}{t^{4} +(v - 8)t^{3} + (-6v + 18)t^{2} + (11v - 8)t + (-6v - 7)}\]

\[\pi_5(t)=\tfrac{t^4+\alpha^2}{t^2}, \pi_{5,v}(t)=\tfrac{(-3v\alpha^2 + v^3)\alpha^2t^4 + (8\alpha^2 - 4v^2)\alpha^2t^3 + 6v\alpha^2t^2 - 8\alpha^2t +v}{\alpha^4t^4 - 2v\alpha^2t^3 + (2\alpha^2 + v^2)t^2 - 2vt + 1}\]

\[\pi_6(t)=\tfrac{t^4 + 2t^2 + 1}{t^3-t}\]
\[\pi_{6,v}(t)=\tfrac{(-25v^3 + 160v^2 - 256v)t^4 + (40v^3 - 208v^2 + 256v)t^3+ (-26v^3 + 96v^2 - 64v)t^2+(8v^3 - 16v^2)t-v^3}{(6v^3 - 37v^2 + 64v - 64)t^4+(-11v^3 + 56v^2 - 32v)t^3+ (6v^3 - 30v^2)t^2+(-v^3 +8v^2)t-v^2}.\]

\subsection{Reading the tables} We now explain how to read tables in \cref{THM:FIRST} and \cref{THM:SECOND} by giving an example. For the label 2A-2A we look for the map $\pi(t) \in \Q(t)$ given in Table $1$ \cite{2105.14623} which is $t^2+1728.$ We then compute a map $J \in \Q(t)$ such that $\pi(t)=J \circ \pi_1(t).$ For 2A-2A, $J(t)=t+1728.$ The morphisms $\pi_G$ for groups $G$ satisfying conditions \ref{conditions} and containing $-I$ for which $G^t=\rho_E(\G_{\Q})$ for some non-CM elliptic curve $E$ defined over $\Q$ are $J \circ \pi_{1,v}=vt^2+1728$ for $v$ squarefree integers not equal to $\pm 1$ and $J(v) \notin \{0,1728,\infty\}.$ For $v=-1$, the group $G$ satisfying conditions \ref{conditions} and containing $-I$ with the morphism $\pi_G$ given by $-t^2+1728$ has the possibility that for some index $2$ subgroup $H$ of $G$ such that $\pm H=G$ there exists a non-CM elliptic curve $E$ defined over $\Q$ satisfying $H^t=\rho_E(\G_{\Q}).$

\begin{remark}

In the column where we give constraints on $v$, we always assume that $v$ avoids those values where the corresponding map $J(v)$ is equal to $0$, $1728$ or $\infty.$
\end{remark}

Before stating our main theorems, we set some notation. Let $M$ be a positive integer. We denote the field $\Q(\zeta_M)$ by $K_M$ and the compositum of all extensions $\{K_{M^n}\}$ for $n \ge 1$ by $K_{M^{\infty}}.$ 

\begin{theorem}\label{THM:FIRST}
 
If $G$ is an open subgroup of $\GL_2(\Zhat)$ satisfying conditions \ref{conditions}, containing $-I$ and there exists an elliptic curve $E$ over $\Q$ such that $\rho_E(\G_{\Q})=G^t$, then $(X_G,\pi_G)$ is isomorphic to exactly one of the following entries in the following table. Conversely, for each group listed in the table below there exists a non CM elliptic curve $E$ over $\Q$ such that $\rho_E(\G_{\Q})=G^t.$

\begin{table}[H] 
\begin{small}
\begin{tabular}{|p{1.5 cm}|p{1.5 cm}|p{2 cm}|p{12 cm}|}\hline 

{Label} & {$\alpha$} & {$\pi_{i,v}$} & {$v$} \\ \hline 

2A-2A   &    {}        & 1 &  Squarefree integer not equal to $\pm 1$\\ \hline

4D-4A   &    {-1/2}  &  2& $v^2+2 \notin (\Q^{\times})^2$, $\Q(\sqrt{v^2+2}) \intersect K_{2^{\infty}}=\Q$ \\ \hline

4E-4A  & {1} & 5& $\Q(\sqrt{\frac{v+\sqrt{v^2-4}}{2}})$ is an extension of degree $4$ over $\Q$, \\ {}&{}&{}& $\Q(\sqrt{\frac{v+\sqrt{v^2-4}}{2}}) \intersect K_{2^{\infty}}=\Q$ \\ \hline
6A-6A   & {} & {}& {} \\ \hline

8E-16A  & {2} & 2& $v^2-8 \notin (\Q^{\times})^2$, $\Q(\sqrt{v^2-8}) \intersect K_{2^{\infty}}=\Q$ \\ \hline

8E-16B  & {1/2} & 2& $v^2-2 \notin (\Q^{\times})^2$, $\Q(\sqrt{v^2-2}) \intersect K_{2^{\infty}}=\Q$ \\ \hline

10A-10A & {} & 1& Squarefree integer not equal to $\pm 1$\\ \hline

10E-10A   & {} & {}& {} \\ \hline

16F-32A  & {1/2} & 2& $\Q(\sqrt{v^2-2}) \intersect K_{2^{\infty}}=\Q$ \\ \hline

16F-32B  & {2} &  2& $\Q(\sqrt{v^2-8}) \intersect K_{2^{\infty}}=\Q$ \\ \hline

\end{tabular}
\end{small}

\end{table}

\end{theorem}

\begin{theorem}\label{THM:SECOND}
 If $E$ is an elliptic curve defined over $\Q$ without CM such that $-I \notin \rho_E(\G_{\Q})$ and $\pm \rho_E(\G_{\Q})$ has genus 0, then $\pm \rho_E(\G_{\Q})$ is isomorphic to exactly one of the following entries. 

\begin{table}[H] 
\begin{small}
\begin{tabular}{|p{1.5 cm}|p{1.5 cm}|p{2 cm}|p{12 cm}|}\hline 

{Label} & {$\alpha$} & {$\pi_{i,v}$}&{$v$} \\ \hline 

2A-2A & {} & 1& -1 \\ \hline
2C-2A & {} & 1 & $v$ is a squarefree integer not equal to $\pm 1$ \\ \hline
2C-2A & {} & 3&$\Q(\tfrac{\sqrt[3]{2v^3 - 9v^2 + 3\sqrt{3}\sqrt{-v^4 + 6 v^3 - 27 v^2 + 54 v - 81} + 27v-27}}{3 \sqrt[3]{2}}$\\
{}&{}&{}& $-\tfrac{\sqrt[3]{2} (-v^2 + 3 v-9)}{3\sqrt[3]{ 2 v^3-9 v^2 + 3 \sqrt{3} \sqrt{-v^4 + 6 v^3-27 v^2 + 54 v-81} + 27 v-27}})$\\
{}&{}&{}& is a degree $3$ extension of $\Q$ \\ \hline

3C-3A & {-1/3}& 2& $v^2+4/3 \notin (\Q^{\times})^2$, $\Q(\sqrt{v^2+4/3}) \intersect K_{3^{\infty}}=\Q$ \\ \hline

3D-3A & {3} & 2& $v^2-12 \notin (\Q^{\times})^2$, $\Q(\sqrt{v^2-12}) \intersect K_{3^{\infty}}=\Q$ \\ \hline

4A-4A &{}&{}&{} \\ \hline

4B-4A &{} & 1& $v$ is a squarefree integer not equal to $\pm 1$ \\ \hline

4C-4A & {} & 1& $v$ is a squarefree integer not equal to $\pm 1$ \\ \hline

4D-4A & {1/2} & 2& $v \in \{0,1/2\}$ \\ \hline

4G-4A & {2} & 2& $v^2-8 \notin (\Q^{\times})^2$, $\Q(\sqrt{v^2-8}) \intersect K_{2^{\infty}}=\Q$ \\ \hline

4G-4A & {1} & 5&  $\Q(\sqrt{\frac{v+\sqrt{v^2-4}}{2}})$ is an extension of degree $4$ over $\Q$, \\ {}&{}&{}& $\Q(\sqrt{\frac{v+\sqrt{v^2-4}}{2}}) \intersect K_{2^{\infty}}=\Q$ \\ \hline

4G-4A & {} & 4& $\Q(\tfrac{-\sqrt{v^2 + 16}}{4} + \tfrac{1}{2} \sqrt{\tfrac{v^2}{2} - \tfrac{v^3 + 16 v}{2 \sqrt{v^2 + 16}} + 8})$\\
{}&{}&{}& is a degree 4 extension of $\Q$ \\
{}&{} &{}&and its intersection with $K_{2^{\infty}}=\Q$ \\ \hline 

4G-4A & {} & 6& $\Q(\tfrac{-\sqrt{v^2 - 16}}{4} + \tfrac{1}{2} \sqrt{\tfrac{v^2}{2} - \tfrac{v^3 - 16 v}{2 \sqrt{v^2 - 16}}})$\\
{}&{}& {}& is a degree 4 extension of $\Q$ \\
{}&{} & {}& and its intersection with $K_{2^{\infty}}=\Q$ \\ \hline 

4G-4B & {1/2} & 2& $v^2-2 \notin (\Q^{\times})^2$, $\Q(\sqrt{v^2-2}) \intersect K_{2^{\infty}}=\Q$ \\ \hline

4G-4B & {1/2} & 5&  $\Q(\sqrt{\frac{v+\sqrt{v^2-1}}{2}})$ is an extension of degree $4$ over $\Q$, \\ {}&{}&{}& $\Q(\sqrt{\frac{v+\sqrt{v^2-1}}{2}}) \intersect K_{2^{\infty}}=\Q$ \\ \hline

4G-8B & {1/2} & 2& $v^2-2 \notin (\Q^{\times})^2$, $\Q(\sqrt{v^2-2}) \intersect K_{2^{\infty}}=\Q$ \\ \hline

4G-8B & {-1/2} & 2& $v^2+2 \notin (\Q^{\times})^2$, $\Q(\sqrt{v^2+2}) \intersect K_{2^{\infty}}=\Q$ \\ \hline

4G-8B & {1/2} & 5 &  $\Q(\sqrt{\frac{v+\sqrt{v^2-1}}{2}})$ is an extension of degree $4$ over $\Q$, \\ {}&{}&{}& $\Q(\sqrt{\frac{v+\sqrt{v^2-1}}{2}}) \intersect K_{2^{\infty}}=\Q$ \\ \hline

4G-8C & {2} & 2& $v^2-8 \notin (\Q^{\times})^2$, $\Q(\sqrt{v^2-8}) \intersect K_{2^{\infty}}=\Q$ \\ \hline

4G-8C & {1} & 5 &  $\Q(\sqrt{\frac{v+\sqrt{v^2-4}}{2}})$ is an extension of degree $4$ over $\Q$, \\ {}&{}&{}& $\Q(\sqrt{\frac{v+\sqrt{v^2-4}}{2}}) \intersect K_{2^{\infty}}=\Q$ \\ \hline

4G-8C & {} & 4 & $\Q(\tfrac{-\sqrt{v^2 + 16}}{4} + \tfrac{1}{2} \sqrt{\tfrac{v^2}{2} - \tfrac{v^3 + 16 v}{2 \sqrt{v^2 + 16}} + 8})$\\
{}&{}& {}&is a degree 4 extension of $\Q$ \\
{}&{} &{}&and its intersection with $K_{2^{\infty}}=\Q$ \\ \hline 

4G-8C & {} & 6 & $\Q(\tfrac{-\sqrt{v^2 - 16}}{4} + \tfrac{1}{2} \sqrt{\tfrac{v^2}{2} - \tfrac{v^3 - 16 v}{2 \sqrt{v^2 - 16}}})$\\
{}&{}& {}&is a degree 4 extension of $\Q$ \\
{}&{} & {}&and its intersection with $K_{2^{\infty}}=\Q$ \\ \hline 
\end{tabular}
\end{small}
\end{table}

\begin{table}[H] 
\begin{small}
\begin{tabular}{|p{1.5 cm}|p{1.5 cm}|p{2 cm}|p{12 cm}|}\hline 

{Label} & {$\alpha$} & {$\pi_{i,v}$}&{$v$} \\ \hline 

4G-8D & {1/2} & 2& $v^2-2 \notin (\Q^{\times})^2$, $\Q(\sqrt{v^2-2}) \intersect K_{2^{\infty}}=\Q$ \\ \hline
4G-8D & {-1/2} & 2& $v^2+2 \notin (\Q^{\times})^2$, $\Q(\sqrt{v^2+2}) \intersect K_{2^{\infty}}=\Q$ \\ \hline

4G-8D & {1/2} & 5 &  $\Q(\sqrt{\frac{v+\sqrt{v^2-1}}{2}})$ is an extension of degree $4$ over $\Q$, \\ {}&{}&{}& $\Q(\sqrt{\frac{v+\sqrt{v^2-1}}{2}}) \intersect K_{2^{\infty}}=\Q$ \\ \hline

4G-8E & {2} & 2 & $v^2-8 \notin (\Q^{\times})^2$, $\Q(\sqrt{v^2-8}) \intersect K_{2^{\infty}}=\Q$ \\ \hline
4G-8E & {-2} & 2 & $v^2-8 \notin (\Q^{\times})^2$, $\Q(\sqrt{v^2-8}) \intersect K_{2^{\infty}}=\Q$ \\ \hline

4G-8E & {2} & 5 &  $\Q(\sqrt{\frac{v+\sqrt{v^2-16}}{2}})$ is an extension of degree $4$ over $\Q$, \\ {}&{}&{}& $\Q(\sqrt{\frac{v+\sqrt{v^2-16}}{2}}) \intersect K_{2^{\infty}}=\Q$ \\ \hline

4G-8F & {2} & 2 & $v^2-8 \notin (\Q^{\times})^2$, $\Q(\sqrt{v^2-8}) \intersect K_{2^{\infty}}=\Q$ \\ \hline

4G-8F & {} & 6 & $\Q(\tfrac{-\sqrt{v^2 - 16}}{4} + \tfrac{1}{2} \sqrt{\tfrac{v^2}{2} - \tfrac{v^3 - 16 v}{2 \sqrt{v^2 - 16}}})$\\
{}&{}& {}&is a degree 4 extension of $\Q$ \\
{}&{} & {}&and its intersection with $K_{2^{\infty}}=\Q$ \\ \hline 

4G-16A & {} & 4 & $\Q(\tfrac{-\sqrt{v^2 + 16}}{4} + \tfrac{1}{2} \sqrt{\tfrac{v^2}{2} - \tfrac{v^3 + 16 v}{2 \sqrt{v^2 + 16}} + 8})$\\
{}&{}& {}& is a degree 4 extension of $\Q$ \\
{}&{} &{}& and its intersection with $K_{2^{\infty}}=\Q$ \\ \hline 

5A-5A & {} & {} & {} \\ \hline

5F-15A & {-5} & {} & {} \\ \hline

5G-5A  & {5} & 2 &  $v^2-20 \notin (\Q^{\times})^2$, $\Q(\sqrt{v^2-20}) \intersect K_{5^{\infty}}=\Q$ \\ \hline

5H-15A & {-5} & 2 & $\Q(\sqrt{v^2+20}) \intersect K_{15}=\Q$ \\ \hline

6B-6A & {} & {} & {} \\ \hline

6C-6A & {} & 1 & $v$ is a squarefree integer not equal to $\pm 1$ \\ \hline

6E-6A & {} & 1 &  $v$ satisfies $\Q(\sqrt{v}) \intersect K_{6}=\Q$ \\ \hline

6E-12A & {} & 1& $v$ satisfies $\Q(\sqrt{v}) \intersect K_{12}=\Q$  \\ \hline

6E-24A & {} & 1 & $v$ satisfies $\Q(\sqrt{v}) \intersect K_{24}=\Q$  \\ \hline

6E-24B & {} & 1 & $v$ satisfies $\Q(\sqrt{v}) \intersect K_{24}=\Q$  \\ \hline

6J-6A & {} & {} & {} \\ \hline

8E-16A & {2} & 2 & {3} \\ \hline
8E-16B & {1/2} & 2 & {3/2} \\ \hline
8F-8A  & {} & {} & {}\\ \hline
8G-8A  & {2} & 2&  $v^2-8 \notin (\Q^{\times})^2$, $\Q(\sqrt{v^2-8}) \intersect K_{2^{\infty}}=\Q$ \\ \hline

8G-8A  & {1} & 5 &  $\Q(\sqrt{\frac{v+\sqrt{v^2-4}}{2}})$ is an extension of degree $4$ over $\Q$, \\ {}&{}&{}& $\Q(\sqrt{\frac{v+\sqrt{v^2-4}}{2}}) \intersect K_{2^{\infty}}=\Q$ \\ \hline

8G-8A & {} & 4 & $\Q(\tfrac{-\sqrt{v^2 + 16}}{4} + \tfrac{1}{2} \sqrt{\tfrac{v^2}{2} - \tfrac{v^3 + 16 v}{2 \sqrt{v^2 + 16}} + 8})$\\
{}&{}& {}& is a degree 4 extension of $\Q$ \\
{}&{} &{}& and its intersection with $K_{2^{\infty}}=\Q$ \\ \hline

8G-8A & {} & 6 & $\Q(\tfrac{-\sqrt{v^2 - 16}}{4} + \tfrac{1}{2} \sqrt{\tfrac{v^2}{2} - \tfrac{v^3 - 16 v}{2 \sqrt{v^2 - 16}}})$\\

{}&{}& {}&is a degree 4 extension of $\Q$ \\
{}&{} &{}&and its intersection with $K_{2^{\infty}}=\Q$ \\ \hline
\end{tabular}
\end{small}
\end{table}

\begin{table}[H] 
\begin{small}
\begin{tabular}{|p{1.5 cm}|p{1.5 cm}|p{2 cm}|p{12 cm}|}\hline 

{Label} & {$\alpha$} & {$\pi_{i,v}$}&{$v$} \\ \hline 

8H-8A  & {} & 1&$v$ is squarefree integer and $\Q(\sqrt{v}) \intersect K_{2^{\infty}}=\Q$ \\ \hline

8H-8A  & {2} & 5 & $\Q(\sqrt{\frac{v+\sqrt{v^2-16}}{2}})$ is an extension of degree $4$ over $\Q$, \\ {}&{}&{}& $\Q(\sqrt{\frac{v+\sqrt{v^2-16}}{2}}) \intersect K_{2^{\infty}}=\Q$ \\ \hline

8N-16A & {1} & 2& $v^2-4 \notin (\Q^{\times})^2$, $\Q(\sqrt{v^2-4}) \intersect K_{2^{\infty}}=\Q$ \\ \hline

8N-16A & {-1} &  2 & $v^2+4 \notin (\Q^{\times})^2$, $\Q(\sqrt{v^2+4}) \intersect K_{2^{\infty}}=\Q$ \\ \hline

8N-16A & {1} &5& $\Q(\sqrt{\frac{v+\sqrt{v^2-4}}{2}})$ is an extension of degree $4$ over $\Q$, \\ {}&{}&{}& $\Q(\sqrt{\frac{v+\sqrt{v^2-4}}{2}}) \intersect K_{2^{\infty}}=\Q$ \\ \hline

8N-16B & {1/2} & 2& $v^2-2 \notin (\Q^{\times})^2$, $\Q(\sqrt{v^2-2}) \intersect K_{2^{\infty}}=\Q$ \\ \hline

8N-16B & {-1/2} & 2& $v^2+2 \notin (\Q^{\times})^2$, $\Q(\sqrt{v^2+2}) \intersect K_{2^{\infty}}=\Q$ \\ \hline

8N-16B & {1/2} & 5& $\Q(\sqrt{\frac{v+\sqrt{v^2-1}}{2}})$ is an extension of degree $4$ over $\Q$, \\ {}&{}&{}& $\Q(\sqrt{\frac{v+\sqrt{v^2-1}}{2}}) \intersect K_{2^{\infty}}=\Q$ \\ \hline

8N-16C & {} & 1&$v$ is a squarefree integer and $\Q(\sqrt{v}) \intersect K_{2^{\infty}}=\Q$ \\ \hline

8N-16C & {1/2} & 2&$v^2-2 \notin (\Q^{\times})^2$, $\Q(\sqrt{v^2-2}) \intersect K_{2^{\infty}}=\Q$ \\ \hline

8N-16C & {-1/2} & 2&  $v^2+2 \notin (\Q^{\times})^2$, $\Q(\sqrt{v^2+2}) \intersect K_{2^{\infty}}=\Q$ \\ \hline

8N-16C & {1/2} & 5& $\Q(\sqrt{\frac{v+\sqrt{v^2-1}}{2}})$ is an extension of degree $4$ over $\Q$, \\ {}&{}&{}& $\Q(\sqrt{\frac{v+\sqrt{v^2-1}}{2}}) \intersect K_{2^{\infty}}=\Q$ \\ \hline

8N-16D & {} & 1& $v$ is a squarefree integer and $\Q(\sqrt{v}) \intersect K_{2^{\infty}}=\Q$ \\ \hline

8N-16D & {1} & 2& $v^2-4 \notin (\Q^{\times})^2$, $\Q(\sqrt{v^2-4}) \intersect K_{2^{\infty}}=\Q$ \\ \hline

8N-16D & {2} & 2 & $v^2-8 \notin (\Q^{\times})^2$, $\Q(\sqrt{v^2-8}) \intersect K_{2^{\infty}}=\Q$ \\ \hline

8N-16D & {1} & 5& $\Q(\sqrt{\frac{v+\sqrt{v^2-4}}{2}})$ is an extension of degree $4$ over $\Q$, \\ {}&{}&{}& $\Q(\sqrt{\frac{v+\sqrt{v^2-4}}{2}}) \intersect K_{2^{\infty}}=\Q$ \\ \hline

8N-16D & {} & 6& $\Q(\tfrac{-\sqrt{v^2 - 16}}{4} + \tfrac{1}{2} \sqrt{\tfrac{v^2}{2} - \tfrac{v^3 - 16 v}{2 \sqrt{v^2 - 16}}})$\\
{}&{}&{}& is a degree 4 extension of $\Q$ \\
{}&{} & {}&and its intersection with $K_{2^{\infty}}=\Q$ \\ \hline

8N-32A & {-1} & 2&$\Q(\sqrt{v^2+2}) \intersect K_{2^{\infty}}=\Q$ \\ \hline

8N-32A & {}  & 4& $\Q(\tfrac{-\sqrt{v^2 + 16}}{4} + \tfrac{1}{2} \sqrt{\tfrac{v^2}{2} - \tfrac{v^3 + 16 v}{2 \sqrt{v^2 + 16}} + 8})$\\
{}&{}&{}& is a degree 4 extension of $\Q$ \\
{}&{} &{}&and its intersection with $K_{2^{\infty}}=\Q$ \\ \hline

8N-160A & {-1} & 2&$\Q(\sqrt{v^2+2}) \intersect K_{2^{\infty}}=\Q$ \\ \hline

8N-160A & {}  & 4&$\Q(\tfrac{-\sqrt{v^2 + 16}}{4} + \tfrac{1}{2} \sqrt{\tfrac{v^2}{2} - \tfrac{v^3 + 16 v}{2 \sqrt{v^2 + 16}} + 8})$\\
{}&{}& {}&is a degree 4 extension of $\Q$ \\
{}&{} &{}&and its intersection with $K_{2^{\infty}}=\Q$ \\ \hline

\end{tabular}
\end{small}
\end{table}

\begin{table}[H] 
\begin{small}
\begin{tabular}{|p{1.5 cm}|p{1.5 cm}|p{2 cm}|p{12 cm}|}\hline 

{Label} & {$\alpha$} & {$\pi_{i,v}$}&{$v$} \\ \hline 

8N-160B & {-1} & 2&$\Q(\sqrt{v^2+2}) \intersect K_{2^{\infty}}=\Q$ \\ \hline

8N-160B & {}  & 4&$\Q(\tfrac{-\sqrt{v^2 + 16}}{4} + \tfrac{1}{2} \sqrt{\tfrac{v^2}{2} - \tfrac{v^3 + 16 v}{2 \sqrt{v^2 + 16}} + 8})$\\
{}&{}& {}&is a degree 4 extension of $\Q$ \\
{}&{} &{}&and its intersection with $K_{2^{\infty}}=\Q$ \\ \hline

8P-8A & {2} &2& {} \\ \hline
8P-8B & {1/2} &2& {} \\ \hline

9D-9A & {-3} & 2&$v^2+12 \notin (\Q^{\times})^2$, $\Q(\sqrt{v^2+12}) \intersect K_{3^{\infty}}=\Q$ \\ \hline

9F-9A & {} & {}&{} \\ \hline

9H-9A & {3} & 2&$v^2-12 \notin (\Q^{\times})^2$, $\Q(\sqrt{v^2-12}) \intersect K_{3^{\infty}}=\Q$ \\ \hline
10D-10A & {5} & 2& {} \\ \hline

12B-12B & {}  & 1&$\Q(\sqrt{v}) \intersect K_{12^{\infty}}=\Q$ \\ \hline
12B-12D & {}  & 1&$\Q(\sqrt{v}) \intersect K_{12^{\infty}}=\Q$ \\ \hline
12B-24A & {}  & 1&$\Q(\sqrt{v}) \intersect K_{12^{\infty}}=\Q$ \\ \hline
12B-24D & {}  & 1&$\Q(\sqrt{v}) \intersect K_{12^{\infty}}=\Q$ \\ \hline

12F-12A & {-2}  & 2&{} \\ \hline

12H-12A & {}   & 1&$v$ is a squarefree integer and $\Q(\sqrt{v}) \intersect K_{12^{\infty}}=\Q$ \\ \hline

12H-12A & {-1} & 2& $v^2+12 \notin (\Q^{\times})^2$, $\Q(\sqrt{v^2+4}) \intersect K_{12^{\infty}}=\Q$ \\ \hline

12H-12A & {1}  & 5& $\Q(\sqrt{\frac{v+\sqrt{v^2-4}}{2}})$ is an extension of degree $4$ over $\Q$, \\ {}&{}&{}& $\Q(\sqrt{\frac{v+\sqrt{v^2-4}}{2}}) \intersect K_{12^{\infty}}=\Q$ \\ \hline

16G-16A & {} & 1& $v$ is a squarefree integer and $\Q(\sqrt{v}) \intersect K_{2^{\infty}}=\Q$ \\ \hline

16G-16A & {1/2} & 2& $v^2-2 \notin (\Q^{\times})^2$, $\Q(\sqrt{v^2-2}) \intersect K_{2^{\infty}}=\Q$ \\ \hline

16G-16A & {-1/2} & 2& $v^2+2 \notin (\Q^{\times})^2$, $\Q(\sqrt{v^2+2}) \intersect K_{2^{\infty}}=\Q$ \\ \hline

16G-16A & {1/2}  & 5& $\Q(\sqrt{\frac{v+\sqrt{v^2-1}}{2}})$ is an extension of degree $4$ over $\Q$, \\ {}&{}&{}& $\Q(\sqrt{\frac{v+\sqrt{v^2-1}}{2}}) \intersect K_{2^{\infty}}=\Q$ \\ \hline

16G-16C & {} & 1& $v$ is a squarefree integer and $\Q(\sqrt{v}) \intersect K_{2^{\infty}}=\Q$ \\ \hline

16G-16C & {1} & 2&$v^2-4 \notin (\Q^{\times})^2$, $\Q(\sqrt{v^2-4}) \intersect K_{2^{\infty}}=\Q$ \\ \hline

16G-16C & {2} & 2&$v^2-8 \notin (\Q^{\times})^2$, $\Q(\sqrt{v^2-8}) \intersect K_{2^{\infty}}=\Q$ \\ \hline

16G-16C & {1}  & 5& $\Q(\sqrt{\frac{v+\sqrt{v^2-4}}{2}})$ is an extension of degree $4$ over $\Q$, \\ {}&{}&{}& $\Q(\sqrt{\frac{v+\sqrt{v^2-4}}{2}}) \intersect K_{2^{\infty}}=\Q$ \\ \hline

16G-16C & {} & 6&$\Q(\tfrac{-\sqrt{v^2 - 16}}{4} + \tfrac{1}{2} \sqrt{\tfrac{v^2}{2} - \tfrac{v^3 - 16 v}{2 \sqrt{v^2 - 16}}})$\\ \hline
{}&{}&{}& is a degree 4 extension of $\Q$ \\
{}&{} &{}& and its intersection with $K_{2^{\infty}}=\Q$ \\ \hline
\end{tabular}
\end{small}
\end{table}

\begin{table}[H] 
\begin{small}
\begin{tabular}{|p{1.5 cm}|p{1.5 cm}|p{2 cm}|p{12 cm}|}\hline 

{Label} & {$\alpha$} & {$\pi_{i,v}$}&{$v$} \\ \hline 

16G-16E & {2} & 2& $v^2-8 \notin (\Q^{\times})^2$, $\Q(\sqrt{v^2-8}) \intersect K_{2^{\infty}}=\Q$ \\ \hline

16G-16E & {-1} & 2& $v^2+4 \notin (\Q^{\times})^2$, $\Q(\sqrt{v^2+4}) \intersect K_{2^{\infty}}=\Q$ \\ \hline

16G-16E & {} & 6&$\Q(\tfrac{-\sqrt{v^2 - 16}}{4} + \tfrac{1}{2} \sqrt{\tfrac{v^2}{2} - \tfrac{v^3 - 16 v}{2 \sqrt{v^2 - 16}}})$\\
{}&{}&{}& is a degree 4 extension of $\Q$ \\
{}&{} &{}& and its intersection with $K_{2^{\infty}}=\Q$ \\ \hline

16G-16F & {1} & 2&$v^2-4 \notin (\Q^{\times})^2$, $\Q(\sqrt{v^2-4}) \intersect K_{2^{\infty}}=\Q$ \\ \hline

16G-16F & {-1} & 2&$v^2+4 \notin (\Q^{\times})^2$, $\Q(\sqrt{v^2+4}) \intersect K_{2^{\infty}}=\Q$ \\ \hline

16G-16F & {1} & 5& $\Q(\sqrt{\frac{v+\sqrt{v^2-4}}{2}})$ is an extension of degree $4$ over $\Q$, \\ {}&{}&{}& $\Q(\sqrt{\frac{v+\sqrt{v^2-4}}{2}}) \intersect K_{2^{\infty}}=\Q$ \\ \hline

16G-16G & {} & 1& $v$ is a squarefree integer and $\Q(\sqrt{v}) \intersect K_{2^{\infty}}=\Q$ \\ \hline

16G-16G & {1/2} & 2& $v^2-2 \notin (\Q^{\times})^2$, $\Q(\sqrt{v^2-2}) \intersect K_{2^{\infty}}=\Q$ \\ \hline

16G-16G & {-1/2} & 2&$v^2+2 \notin (\Q^{\times})^2$, $\Q(\sqrt{v^2+2}) \intersect K_{2^{\infty}}=\Q$ \\ \hline

16G-16G & {1/2} & 5& $\Q(\sqrt{\frac{v+\sqrt{v^2-1}}{2}})$ is an extension of degree $4$ over $\Q$, \\ {}&{}&{}& $\Q(\sqrt{\frac{v+\sqrt{v^2-1}}{2}}) \intersect K_{2^{\infty}}=\Q$ \\ \hline

16G-16H & {1/2} & 2&$v^2-2 \notin (\Q^{\times})^2$, $\Q(\sqrt{v^2-2}) \intersect K_{2^{\infty}}=\Q$ \\ \hline

16G-16H & {-1/2} & 2& $v^2+2 \notin (\Q^{\times})^2$, $\Q(\sqrt{v^2+2}) \intersect K_{2^{\infty}}=\Q$ \\ \hline

16G-16H & {1/2} &5& $\Q(\sqrt{\frac{v+\sqrt{v^2-1}}{2}})$ is an extension of degree $4$ over $\Q$, \\ {}&{}&{}& $\Q(\sqrt{\frac{v+\sqrt{v^2-1}}{2}}) \intersect K_{2^{\infty}}=\Q$ \\ \hline

16G-32A & {-1} & 2&$\Q(\sqrt{v^2+2}) \intersect K_{2^{\infty}}=\Q$ \\ \hline

16G-32A & {}  & 4&$\Q(\tfrac{-\sqrt{v^2 + 16}}{4} + \tfrac{1}{2} \sqrt{\tfrac{v^2}{2} - \tfrac{v^3 + 16 v}{2 \sqrt{v^2 + 16}} + 8})$\\
{}&{}& {}&is a degree 4 extension of $\Q$ \\
{}&{} &{}&and its intersection with $K_{2^{\infty}}=\Q$ \\ \hline

16G-160A & {-1} & 2&$\Q(\sqrt{v^2+2}) \intersect K_{2^{\infty}}=\Q$ \\ \hline

16G-160A & {}  & 4&$\Q(\tfrac{-\sqrt{v^2 + 16}}{4} + \tfrac{1}{2} \sqrt{\tfrac{v^2}{2} - \tfrac{v^3 + 16 v}{2 \sqrt{v^2 + 16}} + 8})$\\
{}&{}&{}& is a degree 4 extension of $\Q$ \\
{}&{} &{}&and its intersection with $K_{2^{\infty}}=\Q$ \\ \hline

16G-160B & {-1} & 2&$\Q(\sqrt{v^2+2}) \intersect K_{2^{\infty}}=\Q$ \\ \hline
\end{tabular}
\end{small}
\end{table}

\begin{table}[H] 
\begin{small}
\begin{tabular}{|p{1.5 cm}|p{1.5 cm}|p{2 cm}|p{12 cm}|}\hline 

{Label} & {$\alpha$} & {$\pi_{i,v}$}&{$v$} \\ \hline 

16G-160B & {}  &4& $\Q(\tfrac{-\sqrt{v^2 + 16}}{4} + \tfrac{1}{2} \sqrt{\tfrac{v^2}{2} - \tfrac{v^3 + 16 v}{2 \sqrt{v^2 + 16}} + 8})$\\
{}&{}&{}& is a degree 4 extension of $\Q$ \\
{}&{} &{}&and its intersection with $K_{2^{\infty}}=\Q$ \\ \hline

18C-18A & {} & {}&{} \\ \hline

18D-18A & {} & 1& $\Q(\sqrt{v}) \intersect K_{18^{\infty}}=\Q$ \\ \hline

18D-36A & {} &  1&$\Q(\sqrt{v}) \intersect K_{18^{\infty}}=\Q$ \\ \hline
18D-72A & {} &  1&$\Q(\sqrt{v}) \intersect K_{18^{\infty}}=\Q$ \\ \hline
18D-72B & {} &  1&$\Q(\sqrt{v}) \intersect K_{18^{\infty}}=\Q$ \\ \hline
\end{tabular}
\end{small}
\end{table}
\end{theorem}

\begin{remark}
    Converse is false. It is possible that some of the groups listed in above theorem have no index 2 subgroups $H$ that satisfy $\pm H=G.$ For example, group with label 5A-5A is one such.
\end{remark}

\section{Modular Curves} \label{sec:Modular Curves}

In this section, we discuss some background material on modular curves. For further details on modular curves please refer Chapter $4$ of \cite{MR0337993}.

Let $G$ be an open subgroup of $\GL_2(\Zhat)$ such that $\det(G)=\Zhat^{\times}$ and $-I \in G.$ Let $m_G$ be the level of $G.$ We can also think of $G$ as a subgroup of $\GL_2(\ZZ/m_G\ZZ)$.

Let $k$ be a number field. Let $N$ be a positive integer. For an elliptic curve $E$ defined over $k$, let $E[N] \subseteq E[\kbar]$ denote the $N$-torsion subgroup of $E.$ We define $G$-level structure on $E$ as an $G$-equivalence class of isomorphisms $\alpha \colon E[N] \to (\ZZ/N\ZZ)^2$, where two isomorphisms $\alpha_1$ and $\alpha_2$ are equivalent if there exists a $g \in G$ such that $\alpha_1=g \circ \alpha_2.$ 
Associated to $G$, there is a smooth, compact and geometrically connected curve $X_G$ over $\Q$ that can be defined as the generic fiber of the coarse space for the stack $M_{G}$ which parametrizes elliptic curves with $G$-level structure (refer \cite{MR0337993} chapter $4$, section 3 for details). If $G$ is a subgroup of $G'$, then there is a natural morphism from $X_G$ to $X_{G'}$. In particular, we always have a morphism $\pi_G \colon X_G \to X_{\GL_2(\ZZ/m_G\ZZ} \isom \PP^1.$ 

\begin{definition}
Let $k$ be a number field. 
We define $\Aut_k(X_G,\pi_G)$ as the set of those isomorphisms $f \colon (X_G)_k \to (X_G)_k$ which satisfy $\pi_G \circ f=\pi_G.$
\end{definition}

The following lemma explains how modular curves parametrizes elliptic curves.

\begin{lemma}\label{lemm:moduli}
Let $E$ be a non CM elliptic curve over $k$ such that $j(E) \notin \{0,1728,\infty\}.$ Then, $\rho_E(\Galk)$ is conjugate to a subgroup of $G^t$ if and only if $j(E) \in \pi_G(X_G(k)).$
\end{lemma}

\begin{proof}
See Proposition $3.2$ in \cite{1508.07663}.
\end{proof}

If $-I \notin G$, then $X_G$ as a stack can be defined similarly as before and as coarse spaces $X_G$ and $X_{<G,-I>}$ are isomorphic, i.e., the morphism $\pi_G \colon X_G \to \PP^1$ is same as the morphism $\pi_{<G,-I>} \colon X_{<G,-I>} \to \PP^1.$ However, if $-I \notin G$, then lemma \ref{lemm:moduli} is false. Let $E$ be a non CM elliptic curve over $k$ and let $E_d$ be one of its quadratic twists for a $d \notin (k^{\times})^2$. Since $-I \notin G$, $\rho_{E_d}(\Galk)$ need not be a subgroup of $G^t.$ Hence, whether $\rho_E(\Galk)$ is a subgroup of $G^t$ or not cannot depend only on the $j$-invariant of $E.$

\section{Galois Images over $\Q$}

For the rest of this article, equality and inclusion of groups is up to conjugation.

\begin{definition}\label{def:commutator}
Let $G$ be a profinite group with discrete topology. The commutator subgroup of $G$ is the subgroup of $G$ generated by elements of the form $ghg^{-1}h^{-1}$ for $g,h$ in $G.$ We will denote the closure of commutator subgroup of $G$ by $[G,G].$
\end{definition}

Let $G$ be a subgroup of $\GL_2(\Zhat)$ satisfying conditions \ref{conditions}. The associated modular curve $X_G$ has a rational point, i.e., $X_G(\Q) \ne \emptyset.$ This is equivalent to saying that there exist infinitely many elliptic curves $E$ defined over $\Q$ such that $\rho_E(\G_{\Q})$ is a subgroup of $G^t$ in $\GL_2(\Zhat).$ However, it is possible that there are no elliptic curves $E$ such that $\rho_E(\G_{\Q})$ is $G^t.$ As an example, consider $\GL_2(\Zhat)$. For an elliptic curve $E$ defined over $\Q$, $\rho_E(\G_{\Q})$ is a subgroup of $\GL_2(\Zhat)$ but there is no elliptic curve $E$ over $\Q$ such that $\rho_E(\G_{\Q})$ is equal to $\GL_2(\Zhat)$ because a necessary condition for a subgroup $G$ to be equal to $\rho_E(\G_{\Q})$ for some elliptic curve $E$ defined over $\Q$ is that $[G,G]=G \intersect \SL_2(\Zhat).$ For $\GL_2(\Zhat)$ we can verify that $[G,G]$ is an index $2$ subgroup of $\SL_2(\Zhat).$ In the following proposition we show that this necessary condition is sufficient under certain assumptions. 


\begin{proposition}\label{commutator condition} Let $G$ be a genus $0$ open subgroup of $\GL_2(\Zhat)$ satisfying conditions in \cref{conditions} and containing $-I$. There exists an elliptic curve $E$ defined over $\Q$ such that the image of $\rho_E$ is equal to $G^t$ in $\GL_2(\Zhat)$ if and only if $[G^t,G^t]=G^t \intersect \SL_2(\Zhat).$
\end{proposition}

\begin{proof}

    Assume that there exists an elliptic curve $E$ defined over $\Q$ such that the image of $\rho_E$ is equal to $G^t.$ Then, \[[G^t,G^t]=G^t \intersect \SL_2(\Zhat)\] follows from \cref{eq:commthick}.

We will now show that if $[G^t,G^t]=G^t \intersect \SL_2(\Zhat)$, then there exists an elliptic curve $E$ defined over $\Q$ such that the image of $\rho_E$ is equal to $G^t.$ It is enough to show that there exists an elliptic curve $E$ defined over $\Q$ such that $[\rho_E(\G_{\Q}),\rho_E(\G_{\Q})]=[G^t,G^t]$ because if such an $E$ exists, then \[[G^t:\rho_E(\G_{\Q})]=[G^t \intersect \SL_2(\Zhat):\rho_E(\G_{\Q}) \intersect \SL_2(\Zhat)]\] using \cref{eq:detsurjective} and \[[G^t \intersect \SL_2(\Zhat):\rho_E(\G_{\Q})\intersect \SL_2(\Zhat)]=[[G^t,G^t]:[\rho_E(\G_{\Q}),\rho_E(\G_{\Q})]=1.\] The existence of such an $E$ follows directly from Theorem $2.11$ \cite{MR3350106} that states there are infinitely many curves $E$ over $\Q$ satisfying $\rho_E(\G_{\Q}) \subseteq G^t$ and $[\rho_E(\G_{\Q}), \\ \rho_E(\G_{\Q})]=[G^t,G^t].$ Hence, proved.
\end{proof}

\begin{proposition}\label{index 2 cond}
Let $G$ be a subgroup of $\GL_2(\Zhat)$ satisfying conditions \ref{conditions} and containing $-I$ such that $[G^t,G^t] \subsetneq G^t \intersect \SL_2(\Zhat).$  If there exists an elliptic curve $E$ defined over $\Q$ such that $\pm \rho_E(\G_{\Q})=G^t$, then $[G^t,G^t]$ is an index $2$ subgroup of $G^t \intersect \SL_2(\Zhat).$
\end{proposition}

\begin{proof}
 Assume that there exists an elliptic curve $E$ over $\Q$ such that  $\pm \rho_E(\G_{\Q})=G^t$. We have assumed that $[G^t,G^t]$ is a proper subgroup of $G^t \intersect \SL_2(\Zhat)$ so $-I \notin \rho_E(\G_{\Q})$ and hence, $\rho_E(\G_{\Q})$ is an index $2$ subgroup of $G^t.$ We observe that \[[G^t,G^t]=[\pm \rho_E(\G_{\Q}),\pm \rho_E(\G_{\Q})]=[\rho_E(\G_{\Q}),\rho_E(\G_{\Q})]=\rho_E(\G_{\Q}) \intersect \SL_2(\Zhat).\] Therefore, $[G^t,G^t]$ is an index $2$ subgroup of $G^t \intersect \SL_2(\Zhat).$

\end{proof}

In light of results of this section, we aim to compute groups $G$ containing $-I$ that satisfy conditions \ref{conditions} such that $[G,G]$ is a subgroup of $G \intersect \SL_2(\Zhat)$ of index at most 2. We discuss this in next section.

\section{Family of Groups}

\begin{definition}\label{family of groups}
Let $(H,G,\calA,\psi)$ be a $4$-tuple where $H$ is a normal subgroup of $G$ and $\psi \colon G \to \calA$ is a group homomorphism. For a homomorphism $\phi \colon \calA \to G/H$ define \[H_{\phi}:=\{g \in G~|~gH=\phi(\psi(g))\}.\] 
\end{definition}

\begin{lemma}
The set $H_{\phi}$ as defined in definition \ref{family of groups} is a subgroup of G.
\end{lemma}

\begin{proof}
The identity element of $G$ lies in $H_{\phi}.$ If $g_1$ and $g_2$ are in $H_{\phi}$, then \[(g_1g_2)H=\phi(\psi(g_1))\phi(\psi(g_2))=\phi(\psi(g_1g_2)).\] If $g \in H_{\phi}$, then \[g^{-1}H=\phi(\psi(g))^{-1}=\phi(\psi(g^{-1})).\] Hence, shown.
\end{proof}

\begin{definition}
The collection of groups $H_{\phi}$ as $\phi$ varies over homomorphisms from $\calA$ to $G/H$ is called the family of groups generated by $(H,G,\calA,\psi)$. 
\end{definition}

\begin{lemma}
\label{lemma:commutator dissolve}
Assume the following.
\begin{itemize}
    \item The group $\calA$ is abelian.
    
    \item The quotient group $G/H$ is abelian.
    
    \item The group $H_{\phi}$ is the fibered product of $G$ and $\calA$ with respect to the homomorphisms $\eta$ and $\phi.$
\end{itemize}
Then $[H_{\phi},H_{\phi}]=[G,G].$
\end{lemma}

\begin{proof}
Since $\calA$ is abelian, $[\calA,\calA]$ is trivial.
The lemma now follows from lemma $1$ given in Part $3$, section $3$ of \cite{MR0568299}.
\end{proof}

We will now discuss family of groups in the context of modular curves, i.e., we will assume that $G$ is an open subgroup of $\GL_2(\Zhat)$, hence of finite index.

Let $(G,G_0,\Zhat^{\times},\det)$ be a $4$-tuple as in definition \ref{family of groups}. We also assume that the groups $G$, $G_0$ satisfy conditions \ref{conditions}, contain $-I$ and the quotient group $G_0/G$ is finite and abelian. For a homomorphism $\phi \colon \Zhat^{\times} \to G_0/G$, let $G_{\phi}$ be a member of the family generated by $(G,G_0,\Zhat^{\times},\det)$, i.e.,   \[G_{\phi}:=\{g \in G_0~|~gG=\phi(\det(g))\}.\]

From \cite{2105.14623}, Theorem $1.2$ we know that there exists finitely many pairs $(G^i,G^i_0)$ for $i \in \{1,\ldots,269\}$ where $G^i$ is a normal subgroup of $G^i_0$, $G^i_0/G^i$ is finite, abelian, groups $G^i,G^i_0$ satisfy conditions \ref{conditions} and contain $-I$ such that the following statement holds. If $G$ is a group satisfying \ref{conditions} and contains $-I$ then there exists a pair $(G^i,G^i_0)$ and a homomorphism $\phi \colon \Zhat^{\times} \to G^i_0/G^i$ such that $G$ is of the form $G_{i_{\phi}}.$ We would like to mention that even for a fixed pair $(G^i,G^i_0)$ there can be infinitely many homomorphisms $\phi.$ We discuss this in detail in the following paragraph. 

Fix a pair $(G^i,G^i_0)$ as above. Let $v$ be a parameter that takes values in $\Q$. Fix a homomorphism $\gamma_v \colon \Gal(\Qab/\Q) \to G^i_0/G^i.$ Define $\phi_v \colon \Zhat^{\times} \to G^i_0/G^i$ satisfying $\phi_v \circ \chi_{\cyc}=\gamma_v.$ If $G$ is a group satisfying conditions \ref{conditions} then there exists a pair $(G^i,G^i_0)$ and a $v \in \Q$ such that $G$ is equal to  $G_{i_{\phi_v}}.$

For rest of this section, fix one such pair $(G^i,G^i_0)$. We will discuss computation of commutator group of $G_{i_{\phi_v}}$ as $v$ varies in $\Q.$

\begin{remark}
The parameter $v$ is not allowed to take all values in $\Q.$ Explicitly, $v$ avoids at least those values in $\Q$ for which the Galois cover $X_{G^i_0} \to \PP^1$ that can be described explicitly by a rational function $J_i \in \Q(t)$ is ramified, i.e, we avoid those values of $v$ such that $J_i(v) \in \{0,1728,\infty\}.$
\end{remark}

Since $G^i_0/G^i$ is finite and abelian, there exists a finite abelian extension $K_v$ of $\Q$ such that $\gamma_v$ factors through $K_v.$ By the Kronecker-Weber theorem, we know that $K_v$ is contained in some cyclotomic extension of $\Q.$ Let $M$ be the smallest positive integer such that $K_v \subseteq K_{M_v}.$ The level of $G_{\phi_v}$ then is the least common multiple of level of $G^i$ and $M_v.$ There are finitely many $v \in \Q$ satisfying the condition that every prime divisor of $M_v$ is also a prime divisor of level of $G^i.$ This is shown in lemma $8.1$ \cite{2105.14623}. For all other values of $v$ we use the following lemma to compute their commutator groups. 

\begin{lemma}
Assume the notation from above.
Suppose there exists a prime $p$ such that $p$ divides $M_v$ but $p$ does not divide the level of $G^i.$ Then, $[G_{\phi_v},G_{\phi_v}]=[G^j_0,G^j_0].$ 
\end{lemma}

\begin{proof}
We can write $G_{\phi_v}$ as the fibered product of $G^j_0 \times \prod_{l \nmid~level ~of ~G^j_0} \ZZ_l^{\times}$ with respect to the homomorphisms $\eta \colon G^j_0 \to G^j_0/G^i$ and $\phi_v \colon \prod_{l \nmid~level ~of ~G^j_0} \ZZ_l^{\times} \to G^j_0/G^i.$ The proof now follows from \cref{lemma:commutator dissolve}.
\end{proof}

\subsection{Proof of \cref{THM:FIRST}}
Let $S$ be the set of labels $\{4G-4A,4G-4B,4G-8A\emph{-}4G-8F,4G-16A,5H-5A,5H-15A,6E-6A,6E-12A,6E-24A,6E-24B,6K-6A,6K-24A,6L-6A,6L-12A,6L-24A,6L-24B,8B-8A\emph{-}8B-8D,8E-16A,8E-16B,8N-8A\emph{-}8N-160B,8P-8A,8P-8B,12B-12A\emph{-}12B-24D,12C-12A\emph{-}12C-24D,12G-12A\emph{-}12G-24B,12I-12A\emph{-}12I-24D,16B-16A\emph{-}16C-16D,16F-32A\emph{-}16G-240C,18B-18A\emph{-}18B-72D,18D-18A\emph{-}18D-72B,24A-24A\emph{-}24B-24H,25B-25A,25B-25B,28A-28A\emph{-}48A-48H\}.$ 

Define $b_0$ as the least common multiple of orders of elements $a \in \Aut_{K_N}(X_G,\pi_G)$ that divide some power of $N$. Define $b$ as $2b_0$ if $N$ is congruent to $2$ modulo $4$ otherwise define $b$ as $b_0.$

We fix a pair $(X_G,\calA)$ where $X_G$ is one of the modular curves listed in Tables $1$-$9$ \cite{2105.14623} and $\calA$ is a finite abelian subgroup of $\Aut_{\Q}(X_G,\pi_G)$. Fix a function $u \in \Q(t)$ satisfying $\calA=\{f \in \Qbar(t) | u \circ f=u\}.$ We then compute $J \in \Q(t)$ such that $J \circ u=\pi_G.$ Let $N$ be the level of $G.$ We then find the corresponding group $G_0$ that describes the family generated by $(G,G_0,\Zhat^{\times},\det)$ as follows. We first compute all the groups $G'$ such that $G^{'} \intersect \SL_2(\Zhat)$ is a supergroup of $G \intersect \SL_2(\Zhat)$ of index $|\calA|$, level of $G^{'}$ is a divisor of level of $G$ and $X_{G^{'}}$ is isomorphic to $\PP^1_{\Q}$ along with corresponding morphism $\pi_{G'} \to \PP^1_{\Q}.$ Among these finitely many groups we find $G'$ such that $\pi_{G'}$ that is equivalent to $u$, i.e., $u(t)=\pi_{G'} \circ g$ for some degree 1 rational function $g \in \Q(t).$ The group $G'$ is our $G_0$ that we were interested in. If the label of $X_G$ is not in $S$ we also compute all the other modular curves $(X_{G'},\pi_{G'})$ and generators for $G'$ such that $X_{G'}$ is $\PP^1_{\Q}$ and level of $G'$ divides $Nb$ and is a multiple of $N.$ 

For each computed curve, we compute the index of its commutator subgroup using the function \textit{FindCommutatorSubgroup} in \textit{PossibleIndices/IndexComputations.m} which is based on the method discussed in Lemma $4.11$ and Lemma $4.12$ \cite{1508.07663}. 

The only curves $(X_G,\pi_G)$ for which $[G,G]=G \intersect \SL_2(\Zhat)$ or $[G,G]$ is an index $2$ subgroup of $G \intersect \SL_2(\Zhat)$ are the ones listed in \cref{THM:FIRST} and \cref{THM:SECOND} respectively.

\section{Some Group theory}

In this section we will discuss some concepts from group theory that will be used to prove a result about surjectivity of adelic representations associated to elliptic curves over number fields other than $\Q.$ This section can be read independently from other sections.

\begin{definition}
For a group $G$, we denote the set of isomorphism classes of finite, non abelian, simple quotients of $G$ by $\Quo(G).$ 
\end{definition}

\begin{lemma}\label{lemm:Quosbgp}
Let $G$ be a group. Let $H$ be an open subgroup of $G.$ Let $N_H \in \Quo(H)$. Then, there exists an element $N_G \in \Quo(G)$ such that the order of $N_H$ divides the order of $N_G.$
\end{lemma}

\begin{proof}

Let \[1=G_0 \normsbgp G_1 \normsbgp \ldots \normsbgp G_n=G\] be the composition series of $G$. If we intersect each group in the above sequence with $H$ then we get the following series for $H$
\[1=H_0 \normsbgp G_1 \intersect H=H_1 \normsbgp \ldots \normsbgp G_n\intersect H=G \intersect H=H,\] where each factor group $H_{i+1}/H_{i}$ is isomorphic to a subgroup of $G_{i+1}/G_i.$ We can refine the series for $H$ to the composition series and $N_H$ must appear as one of the factor groups. The proof now follows from observing that the order of each factor group in the composition series of $H$ divides the order of some factor group $H_{i+1}/H_{i}$ which divides the order of $G_{i+1}/G_i$. 
\end{proof}

If $G$ is the inverse limit of profinite groups $\{G_i\}_{i \in \calI}$ and each map $G \to G_i$ is surjective, then $\Quo(G)=\Union_{i \in \calI}\Quo(G_i).$ If $G$ is the direct product of profinite groups $\{G_i\}_{i \in \calI}$, then $\Quo(G)=\Union_{i \in \calI}\Quo(G_i).$

Let $G$ be an open subgroup of $\GL_2(\Zhat).$ Then, $G=G_M \times \prod_{l \nmid M} \GL_2(\ZZ_l)$ for some positive integer $M.$ We will assume that the smallest prime number that does not divide $M$ is larger than every prime that divides $M.$ This can be done by taking $M$ as some suitable multiple of the level of $G.$ With this notation, we are ready to state the following lemma.

\begin{lemma}\label{lemma:intersectionempty}

Let $G=G_M \times \prod_{l \nmid M} \GL_2(\ZZ_l).$ Then

\begin{romanenum}
    \item \label{item:1}For every pair of primes $l \neq l'$ where $l$ and $l'$ do not divide $M$ the intersection $\Quo(G_l) \intersect \Quo(G_{l'})$ is empty.
    
    \item \label{item:2}For every prime $l$ that does not divide $M$, $\Quo(G_M) \intersect \Quo(\GL_2(\ZZ_l))$ is empty.
    
\end{romanenum}
 
\end{lemma}

\begin{proof}

Let us show part \ref{item:1}. For any positive integer $n$, the only non abelian simple quotient of $\GL_2(\ZZ/l^n\ZZ)$ is $\PSL_2(\ZZ/l\ZZ).$ So, $\Quo(\GL_2(\ZZ_l))=\Union \Quo(\GL_2(\ZZ/l^n\ZZ))=\{\PSL_2(\ZZ/l\ZZ)\}.$ For any two distinct primes $\PSL_2(\ZZ/l\ZZ)$ is not equal to $\PSL_2(\ZZ/l'\ZZ).$
Let us show \ref{item:2}. From \cref{lemm:Quosbgp}, we know that if $N_G \in \Quo(G_M)$ then the order of $N_G$ has to divide the order of $\PSL_2(\ZZ/p\ZZ)$ where $p$ is a prime factor of $M.$ If $\PSL_2(\ZZ/r\ZZ)$ is isomorphic to a subgroup of $\PSL_2(\ZZ/p\ZZ)$ then $r$ has to be less than or equal to $p.$ So, the lemma follows. 

\end{proof}

The following theorem is a generalization of Theorem $1.1$ \cite{MR2778661}.

\begin{theorem}\label{prop:H=G}
Let $G$ be an open subgroup of $\GL_2(\Zhat)$ with $\det(G)=\Zhat^{\times}$. We can write $G$ as $G=G_M \times \prod_{l \nmid M} \GL_2(\ZZ_l)$ where the smallest prime number that does not divide $M$ is larger than every prime that divides $M$. Let $H$ be a closed subgroup of $G$. Then, $H=G$ if and only if the following maps are surjective.

\begin{romanenum}
    \item \label{proj1}The projections $\pi_M \colon H \to G_M$ and $\pi_l \colon H \to \GL_2(\ZZ_l)$ where $l$ is a prime that does not divide $M.$
    
    \item \label{proj2}The quotient map $\eta \colon H \to G/[G,G].$
\end{romanenum}
\end{theorem}

\begin{proof}
If $H=G$, then part \ref{proj1} and part \ref{proj2} follow easily. Suppose now that part \ref{proj1} and part \ref{proj2} hold. Assume that $H$ is a proper subgroup of $G.$ We can choose a maximal closed proper subgroup $H'$ of $G$ that contains $H.$ Using Proposition $2.5$ \cite{MR2778661} and lemma \ref{lemma:intersectionempty} we can conclude that $H'$ contains commutator of $G$ but this is a contradiction because part \ref{proj2} implies that the quotient map from $K$ to $G/[G,G]$ is surjective. So, $H=G.$
\end{proof}

We state a corollary to Theorem \ref{prop:H=G} below.

\begin{corollary}
Let $G$ be an open subgroup of $\GL_2(\Zhat)$ satisfying $[G,G]=G \intersect \SL_2(\Zhat).$ Let $E$ be a non CM elliptic curve defined over a number field $k$ such that $k$ is not equal to $\Q$ and $k \intersect \Q^{cyc}=\Q$ such that $\rho_E(\Galk)$ is a subgroup of $G.$ Then, $\rho_E(\Galk)$ is equal to $G$ if and only if the projections $\pi_M \colon \rho_E(\Galk) \to G_M$ and $\pi_l \colon \rho_E(\Galk) \to \GL_2(\ZZ_l)$ are surjective.
\end{corollary}

\begin{proof}

The quotient group $G/[G,G]$ is isomorphic to $\Zhat^{\times}.$ The assumption $k \intersect \Q^{cyc}=\Q$ on $k$ implies that $\det(\rho_E(\Galk))$ is equal to $\Zhat^{\times}.$ The proof now follows from Theorem \ref{prop:H=G}.  
\end{proof}


\begin{bibdiv} 
\begin{biblist}
\bib{MR0337993}{article}{
   author={Deligne, P.},
   author={Rapoport, M.},
   title={Les sch\'{e}mas de modules de courbes elliptiques},
   language={French},
   conference={
      title={Modular functions of one variable, II},
      address={Proc. Internat. Summer School, Univ. Antwerp, Antwerp},
      date={1972},
   },
   book={
      publisher={Springer, Berlin},
   },
   date={1973},
   pages={143--316. Lecture Notes in Math., Vol. 349},
   review={\MR{0337993}},
}

\bib{MR2778661}{article}{
    AUTHOR = {Greicius, Aaron},
     TITLE = {Elliptic curves with surjective adelic {G}alois
              representations},
   JOURNAL = {Experiment. Math.},
  FJOURNAL = {Experimental Mathematics},
    VOLUME = {19},
      YEAR = {2010},
    NUMBER = {4},
     PAGES = {495--507},
      ISSN = {1058-6458},
   MRCLASS = {11G05 (11F80)},
  MRNUMBER = {2778661},
MRREVIEWER = {N\'{u}ria Vila},
       DOI = {10.1080/10586458.2010.10390639},
       URL = {https://doi.org/10.1080/10586458.2010.10390639},
}

\bib{MR3350106}{article}{
    AUTHOR = {Jones, Nathan},
     TITLE = {{${\rm GL}_2$}-representations with maximal image},
   JOURNAL = {Math. Res. Lett.},
  FJOURNAL = {Mathematical Research Letters},
    VOLUME = {22},
      YEAR = {2015},
    NUMBER = {3},
     PAGES = {803--839},
      ISSN = {1073-2780},
   MRCLASS = {11G05},
  MRNUMBER = {3350106},
MRREVIEWER = {Hang Xue},
       DOI = {10.4310/MRL.2015.v22.n3.a10},
       URL = {https://doi.org/10.4310/MRL.2015.v22.n3.a10},
}

\bib {MR0568299}{book}{
    AUTHOR = {Serge Lang and Hale Trotter},
     TITLE = {Frobenius distributions in {${\rm GL}\sb{2}$}-extensions},
    SERIES = {Lecture Notes in Mathematics, Vol. 504},
      NOTE = {Distribution of Frobenius automorphisms in
              ${{\rm{G}}L}\sb{2}$-extensions of the rational numbers},
 PUBLISHER = {Springer-Verlag, Berlin-New York},
      YEAR = {1976},
     PAGES = {iii+274},
   MRCLASS = {12A50 (10K05)},
  MRNUMBER = {0568299},
MRREVIEWER = {G. Frey},
}

\bibitem[Rak21]{2105.14623}{
Rakvi,
\newblock \textit{A Classification of Genus 0 Modular Curves with Rational Points}, 2021;
\newblock arXiv:2105.14623.}

\bib{MR0387283}{article}{
   author={Serre, Jean-Pierre},
   title={Propri\'{e}t\'{e}s galoisiennes des points d'ordre fini des courbes
   elliptiques},
   language={French},
   journal={Invent. Math.},
   volume={15},
   date={1972},
   number={4},
   pages={259--331},
   issn={0020-9910},
   review={\MR{0387283}},
   doi={10.1007/BF01405086},
}

\bibitem[Zyw15]{1508.07663}David Zywina,
\newblock \textit{Possible indices for the Galois image of elliptic curves over $\mathbb{Q}$}, 2015;
\newblock arXiv:1508.07663.

\end{biblist}
\end{bibdiv}
\end{document}